\documentclass[11pt,a4paper]{article}
\usepackage[ansinew]{inputenc}
\usepackage{amsthm}
\usepackage{amssymb}
\usepackage{amsmath}
\usepackage{amsfonts}

\usepackage{epsfig}

\usepackage{graphics,graphicx}

\usepackage[ruled, lined]{algorithm2e}

\newtheorem{theorem}[equation]{Theorem}
\newtheorem{corollary}[equation]{Corollary}
\newtheorem{lemma}[equation]{Lemma}
\newtheorem{proposition}[equation]{Proposition}
\newtheorem{example}[equation]{Example}

\newtheorem{conjecture}[equation]{Conjecture}
\theoremstyle{definition}
\newtheorem{definition}[equation]{Definition}
\newtheorem{remark}[equation]{Remark}

\numberwithin{equation}{section}

\newcommand{\N}{{\mathbb{N}}}

\begin{document}

\title{Comparing rankings by means of competitivity graphs: structural properties and computation}


\author{
R.\,Criado$^{1,2}$, E.\,Garc\'{\i}a$^{1,2}$, F.\,Pedroche$^3$ and M.\,Romance$^{1,2}$\\
\\
{\small $^1$\,Departamento de Matem\'{a}tica Aplicada} \\
{\small ESCET - Universidad Rey Juan Carlos} \\
{\small C/ Tulip\'an s/n 28933 M\'ostoles (Madrid), Spain.}\\
{\small regino.criado@urjc.es, esther.garcia@urjc.es, miguel.romance@urjc.es}\\
\\
{\small $^2$\,Centro de Tecnolog\'{\i}a Biom\'edica} \\
{\small Universidad Polit\'{e}cnica de Madrid} \\
{\small 28223 Pozuelo de Alarc\'on (Madrid), Spain.}\\
 \\
{\small $^3$\,Institut de Matem\`{a}tica Multidisciplin\`{a}ria} \\
{\small Universitat Polit\`{e}cnica de Val\`{e}ncia} \\
{\small Cam\'{\i} de Vera s/n. 46022 Val\`{e}ncia, Spain.}\\
{\small pedroche@mat.upv.es}\\
}
\date{}

\maketitle

\begin{abstract}
In this paper we introduce a new technique to analyze families of rankings focused on the study of structural properties of a new type of graphs. Given a finite number of elements and a family of rankings of those elements, we say that two elements {\em compete} when they exchange their relative positions in at least two rankings. This allows us to define an undirected graph by connecting elements that compete. We call this graph a {\em competitivity graph}. We study the relationship of competitivity graphs with other well-known families of graphs, such as permutation graphs, comparability graphs and chordal graphs. In addition to this, we also introduce certain important sets of nodes in a competitivity graph. For example, nodes that compete among them form a {\em competitivity set} and nodes connected by chains of competitors form a {\em set of  eventual competitors}. We analyze these sets and we show a method to obtain sets of eventual competitors directly from a family of rankings.

\end{abstract}

%

 \section{Introduction}\label{sec:intro}

In a recent work (see \cite{Cri13}) we introduced a method to compare full rankings based on the study of structural properties of a new type of graphs. Regarding our proposal, any finite set of rankings leads to a graph called {\em competitivity graph}. In this paper we analyze the theoretical properties of this graph providing some relationships with some already known families of graphs. Therefore, we are interested more in the tool itself -- the {\em competitivity graph} -- than in the original focus of the problem -- the rankings.

Any full ranking of $n$ elements can be considered a permutation of $\{1,2, \dots, n\}$. Note that according to this approach it is not possible to consider rankings with ties of two or more elements. One of the first studies on comparison of full rankings appears in the work of Kendall \cite{Kendall38}, where the counting of interchanges between two given rankings plays an essential role. In that seminal paper the main objective was to introduce a measure of correlation between two rankings -- the so-called Kendall's measure of correlation. Since then, some papers have been devoted to analyze, compare and extend this coefficient. For example, in \cite{DiGr} some metrics to compare rankings can be found.  Given some rankings one may be interested in obtaining a consensus ranking: a ranking that {\em better} gathers the information collected in the rankings. The resulting consensus ranking is also called an aggregated rank, see \cite{LaMe2012}. Many studies are devoted to describe techniques for aggregate rankings with an eye on ranking web pages on the Internet, see. e.g. \cite{Bar}. The most famous technique to aggregate rankings consists of a Markov chain used by the founders of Google \cite{PaBrMoWi}. A different approach consists of finding a permutation that minimizes the crossings over a set of rankings, see \cite{Bie} and \cite{Bie09} for more details on computational complexity.

The concept of {\em competitivity graph} can be considered as a technique of rank aggregation since it is a graph that collects some information (the crossings) between pairs of adjacent rankings. When dealing with only two rankings, the {\em competitivity graph} turns into the well known {\em permutation graph} from graph theory \cite{GrYe}. This concept must not be confused with the concept of {\em planar permutation graph} which is a different mathematical object \cite{Char}, \cite{Go}. In this connection it is necessary to recall that permutation graphs belong to a bigger class of graphs known as {\em comparability graphs}, which is well characterized since the 1960's (see \cite{GiHo} and \cite{Ga}). It is important to highlight that {\em comparability graphs} are also called transitively orientable graphs (TRO) in the sense that an orientation that preserves transitivity can be set in the graph.  Some other papers related to permutation graphs are \cite{PnLeEv}, where an algorithm to detect permutation graphs based on testing whether both a graph and its complement are TRO graphs is presented, \cite{McCo} where an algorithm with a cost of $O(m+n)$  for recognizing permutation graphs is shown, being $m$ and $n$ the edges and vertices of a graph, respectively, \cite{Ko} where a characterization about the connectivity of permutation graphs is provided, and many others (see, for example, \cite{Wang}, \cite{Kr}, \cite {Limo}). Recently, a new characterization of permutations graph is given in \cite{GeRaRa}: the authors introduce the concept of {\em cohesive vertex-set order} based on the connectivity of pairs of vertices to establish a characterization. We will use this latter characterization as a starting point of some of our analysis on {\em competitivity graphs}.

The structure of the paper is the following: after this introduction, section $2$ is devoted to introducing some preliminary results and the main definitions, particularly the one of competitivity graphs. In section $3$ the relationships among competitivity graphs, comparability graphs and other well known families of graphs are studied. Finally, section $4$ is devoted to studying in depth the sets of eventual competitors. Moreover, an algorithm for computing them directly from the rankings is provided.

 \section{Competitivity graphs}\label{sec:basic}

\begin{definition}
Given a finite set of elements ${\mathcal N}=\{1,\dots, n\}$  ($n\in \N$) that we will call {\it nodes} we define a {\it ranking} $c$ of ${\mathcal N}$ as any bijection $c:{\mathcal N}\longrightarrow {\mathcal N}$.
\end{definition}

Since a ranking is a permutation of the elements of ${\mathcal N}$, we can identify rankings with vectors of ${\mathcal N}^n={\mathcal N}\times \cdots \times {\mathcal N}$ in the following way: if we take a ranking $c:{\mathcal N}\longrightarrow {\mathcal N}$, we identify $c\equiv(i_1,\dots, i_n)\in {\mathcal N}^n$, where $c(1)=i_1, c(2)=i_2$, \dots, $c(n)=i_n$.

If $c\equiv(i_1,\dots, i_n)$ is a ranking, then we will write  $i\prec_{c} j$ when node $i$ appears first than node $j$ in the vector of the ranking $c$, i.e., when $c(i)<c(j)$.

\begin{definition}
Given a finite set ${\mathcal R}=\{c_1, c_2,\dots, c_r \}$ of rankings  we say that the pair of nodes $(i,j)\in {\mathcal N}$ {\it compete} if there exist $c_s,c_t\in\{1,2,\ldots, r\}$ such that $i\prec_{c_s} j$ but $j\prec_{c_t} i$, i.e.,   $i$ and $j$ exchange their relative positions between the rankings $c_{s}$ and $c_{t}$.
\end{definition}

The competitiveness between two nodes $i,j\in {\mathcal N}$ is strongly related with the fact that $(i,j)$ is an inversion of a ranking of the family, as the following result shows. Remember that an {\it inversion} in a ranking $c$ is a pair of two nodes $(i,j)$  such that
\[
(i-j)(c^{-1}(i)-c^{-1}(j))<0.
\]

\begin{lemma}\label{equivdef}
Given a finite set ${\mathcal R}=\{c_1, c_2,\dots, c_r \}$ of rankings, the following conditions are equivalent:
\item [{\it (i)}] The pair of nodes $(i,j)$ compete.
\item[{\it (ii)}] There exists  $c_s\in\{c_1,\ldots, c_{r-1}\}$ such that    $i$ and $j$ exchange their relative positions between the rankings $c_{s}$ and $c_{s+1}$.
\item[{\it (iii)}] There exists a relabeling of the nodes such that $c_1={\rm id}\equiv(1,2,\dots,n)$ and some  $c_s\in\{c_2,\ldots, c_r\}$ with an inversion of the form $(i,j)$.
\end{lemma}

\begin{proof}
Clearly {\it (ii)}$\implies${\it (i)}. For {\it (i)}$\implies${\it (iii)} relabel the nodes such that $c_1={\rm id}$; if $i$ and $j$ exchange their relative positions between rankings $c_s$ and $c_t$ either $(i,j)$ is an inversion of $c_s$ or $(i,j)$ is an inversion of $c_t$.

For {\it (iii)}$\implies${\it (ii)}, once the nodes have been relabeled so that $c_1={\rm id}$ and we have an inversion $(i,j)$ in ranking $c_s$, either $i$ and $j$ exchange their relative positions between $c_s$ and $c_{s-1}$, or $c_{s-1}$ also contains the inversion $(i,j)$. The result follows by induction since ${\mathcal R}$ is a finite set.
\end{proof}

If we take a family of rankings ${\mathcal R}=\{c_1, c_2,\dots, c_r \}$ we are going to associate it a graph that gives information about the structure of the competitiveness between nodes.

\begin{definition}
Let ${\mathcal R}=\{c_1, c_2,\dots, c_r \}$ be a family of rankings of nodes ${\mathcal N}=\{1,\dots, n\}$. We define the {\it competitivity graph} of the family of rankings ${\mathcal R}$ as the undirected graph  $G_c({\mathcal R})=({\mathcal N}, E_{\mathcal R})$, where the set of edges $E_{\mathcal R}$ is given by the following rule: there is a link between nodes $i$ and $j$ if $(i,j)$ compete.
\end{definition}

\begin{remark}
This kind of graphs have already been defined and studied in the particular case of two rankings ($r=2$). They are the so-called {\it permutation graphs}, see \cite{PnLeEv}, \cite{Wang},  \cite {Kr}, \cite{Ko}, \cite{Bie}, \cite{Limo}, \cite{GeRaRa}. Moreover, by Lemma \ref{equivdef}, after the relabeling of the nodes such that $c_1={\rm id}$,   the { competitivity graph} of the family of rankings ${\mathcal R}$ is the graph consisting of nodes ${\mathcal N}$ and edges given by the inversions induced by $c_2,\dots, c_r$, i.e., the union of all the possible permutation graphs given by $c_2,\dots, c_r$, where the union of graphs means the classic union of graphs $G_1\cup G_2=(V_1\cup V_2, E_1\cup E_2)$ if $G_1=(V_1,E_1)$ and $G_2=(V_2,E_2)$ (see, for example \cite{Diestel}).
\end{remark}

\begin{example}
If we consider ${\mathcal N}=\{1,\dots, 6\}$, and the following rankings of ${\mathcal N}$:
\begin{align*}
c_1\equiv\enspace &(1,2,3,4,5,6),\\
c_2\equiv\enspace &(1,3,4,2,5,6),\\
c_3\equiv\enspace &(1,2,5,3,4,6),\\
c_4\equiv\enspace &(3,2,6,1,5,4),
\end{align*}
then the competitivity graph $G_c({\mathcal R})=({\mathcal N}, E_{\mathcal R})$ of the family of rankings ${\mathcal R}=\{c_1,c_2,c_3,c_4\}$ is shown in Figure~\ref{fbarco}.

\bigskip
\begin{figure}[ht!]
\centerline{\includegraphics{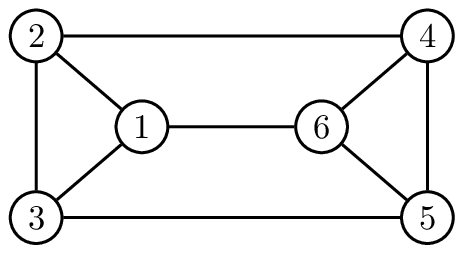}}

\bigskip
\caption{\label{fbarco}
The competitivity graph $G_c({\mathcal R})=({\mathcal N}, E_{\mathcal R})$ of the family of rankings ${\mathcal R}=\{c_1,c_2,c_3,c_4\}$.}
\end{figure}

\bigskip

%
%
\end{example}

\begin{remark}
The computational cost of constructing the competitivity graph $G_c({\mathcal R})=({\mathcal N}, E_{\mathcal R})$ of the family of rankings ${\mathcal R}=\{c_1,\dots ,c_r\}$, where ${\mathcal N}=\{1,\dots, n\}$ is of order $rn^2$, since by Lemma~\ref{equivdef}, after relabeling the nodes such that $c_1={\rm id}$ (which has computational cost of order $rn$ in the worst case), we only have to find the inversions between $c_1$ and each of the $c_s$, which has computational cost of order $(r-1) {n \choose 2}\approx rn^2$.
\end{remark}


One of the goals of this paper is showing that there are some deep connections between structural properties of the competitivity graph $G_c({\mathcal R})$ (such as connectedness, maximal cliques,...) and properties of the rankings and the competitiveness between nodes. As a first example of these relationships we can consider the nodes that compete with a fixed node $i\in{\mathcal N}$.

\begin{definition}
If we take a family of rankings ${\mathcal R}=\{c_1, \dots, c_r\}$ of nodes ${\mathcal N}=\{1,\dots, n\}$ and we fix $i\in {\mathcal N}$, the {\em competitivity set} $C(i)$ {\it of a node} $i$ is the set of elements of ${\mathcal N}$ that compete with $i$ together with $i$, i.e.,
\[
C(i)=\{j\in \mathcal{N}\ |\ (i,j)\hbox{ compete}\}\cup\{i\}.
\]
\end{definition}

\begin{remark}
Local information of the competitivity graph gives properties of the competitivity set of nodes, since it is straightforward to check that $C(i)$ corresponds to $i$ together with all its neighbors in the competitivity graph $G_c({\mathcal R})$.
\end{remark}

\begin{definition}
If we take a family of rankings ${\mathcal R}=\{c_1,\dots, c_r \}$ of nodes ${\mathcal N}=\{1,\dots, n\}$, a set of nodes $C\subseteq{\mathcal N}$ is called a {\it set of competitors} if it is a maximal set with respect to the property of  competition among its elements, i.e., given any two elements  $i,j\in C$, $(i,j)$ compete and $C$ is maximal with respect to this property.
\end{definition}

\begin{remark}
The sets of competitors are exactly the maximal complete subgraphs (maximal cliques) of $G_c({\mathcal R})$. Notice that two nodes compete if and only if they belong to the same set of competitors. Moreover, it can be checked that a set of nodes $C\subseteq{\mathcal N}$ is a competitors set if and only if
\[
C=\bigcap_{i\in C}C(i).
\]
\end{remark}


The previous remark points out the fact that the maximal cliques of $G_c({\mathcal R})$ correspond to the sets of competitors. If we consider other structural sets of nodes of $G_c({\mathcal R})$, such as the connected components of $G_c({\mathcal R})$, we obtain other weaker set of competitors, as the following:

\begin{definition}
If we take a family of rankings ${\mathcal R}=\{c_1, \dots, c_r \}$ of nodes ${\mathcal N}=\{1,\dots, n\}$, we say that a pair of nodes $(i,j)\in {\mathcal N}$ {\it eventually compete} if there exist $k\in\mathbb{N}$ and nodes $i_i,\dots, i_k\in {\mathcal N}$ such that $(i, i_1)$ compete, $(i_1,i_2)$ compete,\dots, and $(i_k,j)$ compete.

A set of nodes $D\subseteq {\mathcal N}$ is called a {\it set of  eventual competitors} if it is a maximal set with respect to the property of eventual competition among its elements.
\end{definition}

\begin{remark}
It is straightforward to check that if a pair of nodes $(i,j)$ compete, then eventually compete. Furthermore, $(i,j)$ eventually compete if and only if $i$ and $j$ are connected by a path in the graph $G_c({\mathcal R})$.

Notice that the sets of eventual competitors of ${\mathcal N}$ are the connected components of $G_c({\mathcal R})$ and two nodes eventually compete if and only if they belong to the same set of eventual competitors. Clearly if two nodes belong to different sets of eventual competitors, they cannot compete.
\end{remark}

In section~\ref{sec:components} we will make a deeper study of sets of eventual competitors, their structure and an algorithm that calculates them without computing the whole competitivity graph $G_c({\mathcal R})$.



 \section{Competitivity, comparability, permutation and chordal graphs}\label{sec:compet-compar-permut}

In this section we will study the relationships among comparability graphs and other well-known families of graphs such as comparability graphs, permutation graphs and chordal graphs. We will follow the notation used in \cite{ISGCtI}.

Let us start this section by reminding the basic well-known definitions of chordal, comparability and permutation graphs that are the main classic families of graphs that will be compared with the competitivity graphs.

\begin{definition}
A graph $G=({\mathcal N},E)$ is {\it chordal} if  each of its cycles of four or more vertices has a chord, which is an edge joining two nodes that are not adjacent in the cycle.
\end{definition}

\begin{definition}
Given any partial ordered set $({\mathcal N}, \preceq)$ we can associate a directed graph $G_{\preceq}$ to  $({\mathcal N}, \preceq)$ as defined in  \cite{GrYe}: the vertex set is ${\mathcal N}$ and there is a link from $i$ to $j$, $i\ne j$, if  $i \preceq j$.

A graph $G=({\mathcal N},E)$ is a {\it comparability graph} if it is the non-directed graph obtained after removing orientation in $G_{\preceq}$ for some partial order $\preceq$ of ${\mathcal N}$.
\end{definition}

\begin{remark}\label{char:compara}
It has been proven \cite{Ghou,PnLeEv} that a graph $G=({\mathcal N},E)$ is a comparability graph if and only if it admits a {\sl transitive orientation of its edges}, i.e. if there is a directed graph $\overrightarrow{G}=({\mathcal N},\overrightarrow{E})$ obtained from $G$ by orienting each edge in $E$, such that if $(i,j),(j,k)\in \overrightarrow{E}$, then $(i,k)\in \overrightarrow{E}$.
\end{remark}

Comparability graphs are a broadly studied class of graphs, and are related to other families of graphs such as permutation graphs, interval graphs, etc. see e.g., \cite{Golum}. We highlight the work of Gallai \cite{Ga} where the modular decomposition of a graph was introduced as a tool to characterize when a graph is a comparability graph.

\begin{definition}
A graph $G=({\mathcal N},E)$ is a {\it permutation graph}  if its vertices represent the elements of a permutation and each one of its edges correspond to a pair of elements that are reversed by the permutation.
\end{definition}

Permutation graphs may also be defined geometrically, as the intersection graphs of line segments whose endpoints lie on two parallel lines.

\begin{remark}
A very useful characterization of permutation graphs is the fact that both the graph $G$ and its complement $\bar G$ (the graph with the same set of nodes and links between nodes that are not linked in $G$) are comparability graphs, i.e., admit a transitive orientation of its edges, see \cite{DuMi}.
\end{remark}

Notice that permutation graphs are both comparability (by using the last remark) and competitivity graphs (simply by considering two rankings: $c_1={\rm id}$ and $c_2$ as the permutation that constitutes the permutation graph).

There are several characterizations of permutation graphs, but we point out the following given in terms of {\it cohesive vertex-set orders} (see \cite{GeRaRa}).

\begin{definition}[\cite{GeRaRa}]\label{def:filipina}
If $G=({\mathcal N},E)$ is a graph, we say that $G$ has a {\it cohesive vertex-set order  (or simply cohesive order)} if there is a relabeling of the nodes such that:
\begin{itemize}
\item[{\it (i)}] if there is a link between nodes $a$ and $b$ and $a<b$, then for every $x$, $a<x<b$, there must exist a link between $a$ and $x$ or a link between $x$ and $b$.
\item[{\it (ii)}] ({\it transitivity}) if there is a link between nodes $a$ and $x$, another between $x$ and $b$, and $a<x<b$, then $a$ and $b$ are also linked.
\end{itemize}
\end{definition}

It was proved the following characterization of permutation graphs in terms of  {\it cohesive vertex-set orders} (see \cite{GeRaRa}):

\begin{theorem}[\cite{GeRaRa}, Theorem 2.3]\label{th:filipino}
$G=({\mathcal N},E)$ is a permutation graph if and only if it has a cohesive vertex-set order.
\end{theorem}

\begin{remark}
Notice that, by using Remark~\ref{char:compara}, for every comparability graphs there is always a relabeling of the nodes such that satisfies condition~{\it(ii)} of Definition~\ref{def:filipina} (transitivity), since it admits a transitive orientation of its vertices (in fact, condition~{\it(ii)} characterizes comparability graphs). Hence we get an alternative proof of the fact that every permutation graph is a comparability graph.
\end{remark}

Following the essence of the definition of cohesive order (Definition~\ref{def:filipina}) and since condition~{\it(ii)} (transitivity) characterizes comparability graphs, we can investigate the graphs that verify condition~{\it(i)}, introducing the following definition.

\begin{definition}\label{basicproperty}
If $G=({\mathcal N},E)$ is a graph,  we say that $G$ has a {\it semi-cohesive vertex-set order (or simply semi-cohesive order)} if there is a relabeling of the nodes such that verifies condition~{\it (i)} of Definition~\ref{def:filipina}, i.e., if there is a link between nodes $a$ and $b$ and $a<b$, then for every $x$, $a<x<b$, there must exist a link between $a$ and $x$ or a link between $x$ and $b$. We will say that a graph $G$ is {\it semi-cohesive} if it has a semi-cohesive vertex-set order.
\end{definition}

While condition~{\it(ii)} of Definition~\ref{def:filipina} is connected with comparability graphs, condition~{\it(i)} (semi-cohesiveness) is connected with competitivity graphs, as the following lemma shows.

\begin{lemma}\label{semicohesive}
Every competitivity graph is semi-cohesive.
\end{lemma}

\begin{proof} Suppose without loss of generality that the ranking ${\rm id}\equiv (1,2,\dots, n)$ belongs to the family of rankings ${\mathcal R}$ that generates the graph. Suppose that there is a link between nodes $a$ and $b$, $a<b$, and take $x$ such $a<x<b$. If there is a link between $a$ and $b$, $(a,b)$ compete, so  there exists another ranking $c_{m}\in {\mathcal R}$ such that $b\prec_{c_{m}}a$. If $x\prec_{c_{m}}a$, the pair $(x,a)$ compete and  there is a link between $a$ and $x$, and otherwise $b\prec_{c_{m}}a\prec_{c_{m}}x$ so the pair $(b,x)$ compete, giving rise to a link between $x$ and $b$.
\end{proof}

\begin{conjecture} Lemma~\ref{semicohesive} is a characterization of competitivity graphs, i.e. $G=({\mathcal N},E)$ is a competitivity graph if and only if it has a semi-cohesive vertex-set order.
\end{conjecture}

In order to show the relationships among competitivity, permutation, comparability and chordal graphs we need to set some auxiliar lemmas.

\begin{definition}
If $G=({\mathcal N},E)$ is a graph, we define the {\it common neighbours} of a pair of nodes $i, j\in{\mathcal N}$  as the set of nodes that are either linked to $i$ or $j$ (or to both).
\end{definition}

It is easy to check that the number of common neighbours of two nodes is controlled in semi-cohesive graphs, simply by using the definition, as the following lemma shows.

\begin{lemma}\label{commoneighbours}
If a graph $G=({\mathcal N},E)$ is semi-cohesive with respect to some labeling of its nodes and we have relabeled  ${\mathcal N}$ according to this labeling, then if $i, j\in {\mathcal N}$ are linked, then the number of the common neighbours of $i$ and $j$ is at least $|i-j|-1$.
\end{lemma}

\begin{lemma}\label{pentagononocomp}
Every cycle of $n$ vertices $C_{n}$ with $n\ge 5$  is not semi-cohesive, and therefore it is not a competitivity graph.
\end{lemma}

\begin{proof}
Let us work with the cycle of 5 vertices. The same argument applies verbatim for cycles of $n$ vertices, $n>5$.
Suppose that it is semi-cohesive and relabel the nodes in a semi-cohesive vertex order. Let us label one of its nodes with 1 and let $i,j,k,m,n$ be the labels of the rest of the nodes as in Figure~\ref{fig:example02}.

\bigskip
\begin{figure}[ht!]
\centerline{\includegraphics{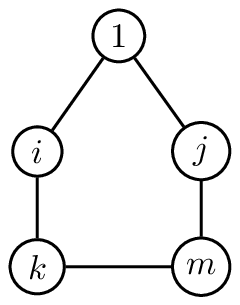}}
\bigskip
\caption{\label{fig:example02} The cycle of 5 nodes $C_5$}
\end{figure}

\bigskip
\noindent The pairs of nodes $(1,i)$ and $(1,j)$ must satisfy Lemma~\ref{commoneighbours} so $i, j\le 4$. If $i=4$ then $j, k\in \{2, 3\}$  because it is semi-cohesive, hence $m=5$. There is a link between $j$ and $m$ and $j<i=4<m$ so using again that it is semi-cohesive there must exist a link between $i$ and $j$ or a link between $i$ and $m$, which is not the case. Hence $i\ne 4$ and by symmetry $j\ne 4$, so one of them is 2 and the other is 3. Without loss of generality suppose that  $i=2$, $j=3$ and hence we are in the situation of Figure~\ref{fig:example03}.

\bigskip
\begin{figure}[h!]
\centerline{\includegraphics{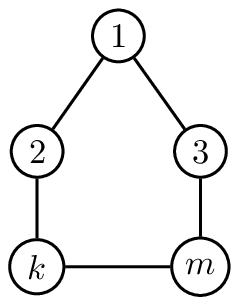}}
\bigskip
\caption{\label{fig:example03} The cycle of 5 nodes $C_5$}
\end{figure}

\bigskip
\noindent  Since $k\ge 4$ and $2$ and $k$ are linked,  since it is semi-cohesive there  must exist a link between $2$ and $3$ or a link between $3$ and $k$, which again is not true and therefore $C_5$ is not semi-cohesive.
\end{proof}

\begin{lemma}
\label{notcomnotcomnotint} Every cycle  with an odd number of vertices $C_{2n+1}$,  $n\ge 2$,  is neither  a competitivity graph, nor a comparability graph nor chordal.
\end{lemma}

\begin{proof}
It is easy to see that these graphs  do not admit a transitive orientation of their vertices, i.e., they are not comparability. Moreover they are not a competitivity graph by Lemma~\ref{pentagononocomp}, and they are not chordal either, by definition.
\end{proof}

In the following result we completely relate the classes of competitivity graphs,  comparability graphs and chordal graphs.

\begin{theorem}\label{contraejemplos}
If ${\mathcal G}$ is the set of all finite graphs, $CH{\mathcal G}$, $P{\mathcal G}$, $C{\mathcal G}$ and $CP{\mathcal G}$ are the sets of all chordal, permutation, comparability and competitivity graphs respectively, then:
\begin{itemize}
 \item[{\it (i)}] There are comparability graphs that are neither competitivity  graphs nor chordal graphs, i.e. $C{\mathcal G}\setminus(CP{\mathcal G}\cup CH{\mathcal G})\ne \emptyset$.
 \item[({\it ii)}] There are competitivity graphs that are neither comparability graphs nor chordal graphs, i.e. $CP{\mathcal G}\setminus(C{\mathcal G}\cup CH{\mathcal G})\ne \emptyset$.
 \item[{\it (iii)}] There are  graphs that are neither comparability, competitivity  graphs nor chordal graphs, i.e. ${\mathcal G}\setminus(C{\mathcal G}\cup CP{\mathcal G}\cup C{\mathcal G})\ne \emptyset$.
 \item[{\it (iv)}] Permutation graphs coincide with the intersection of competitivity and comparability graphs, but not every permutation graph is chordal, i.e. $P{\mathcal G}=C{\mathcal G}\cap CP{\mathcal G}$ and $P{\mathcal G}\not\subseteq CH{\mathcal G}$.
 \item[{\it (v)}] There are chordal graphs that are neither  comparability graphs nor competitivity graphs, i.e. $CH{\mathcal G}\setminus(C{\mathcal G}\cup CP{\mathcal G})\ne \emptyset$.
 \item[{\it (vi)}] There are graphs that are both competitivity and chordal but not comparability, i.e. $CP{\mathcal G}\cap CH{\mathcal G}\cap\left({\mathcal G}\setminus C{\mathcal G}\right)\ne \emptyset$.
 \item[{\it (vii)}] There are graphs that are both comparability and chordal but not competitivity, i.e. $C{\mathcal G}\cap CH{\mathcal G}\cap\left({\mathcal G}\setminus CP{\mathcal G}\right)\ne \emptyset$.
 \item[{\it (viii)}] There are graphs that are chordal, competitivity and comparability, i.e. $CH{\mathcal G}\cap CP{\mathcal G} \cap C{\mathcal G}\ne \emptyset$.
\end{itemize}
\end{theorem}

\begin{proof} {\it (i)} The cycle  of 6 vertices, $C_6$,  is a comparability  graph since we can give a transitive orientation to its edges, as Figure~\ref{fig:example04} shows.

\bigskip
\begin{figure}[ht!]
\centerline{\includegraphics[angle=90,origin=c]{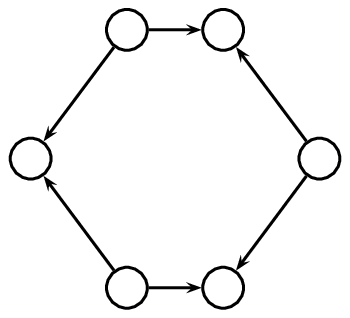}}
\medskip
\caption{\label{fig:example04} The cycle of 6 nodes $C_6$ with a transitive orientation}
\end{figure}

\bigskip
\noindent Nevertheless, $C_6$ is not semi-cohesive by Lemma~\ref{pentagononocomp} and hence it is not a competitivity graph.

\medskip
\noindent{\it (ii)} Consider the  complement of the cycle of 6 vertices $\bar C_6$ with the labeling of its nodes as shown in Figure~\ref{fbarco}. It is a competitivity graph of a family of rankings but its edges do not admit a transitive orientation, hence it is not comparability.

\medskip
\noindent{\it (iii)} Every cycle  with an odd number of vertices $C_{2n+1}$,  $n\ge 2$,  is neither  a competitivity graph, nor a comparability graph nor chordal, as seen in Lemma~\ref{notcomnotcomnotint}.

\medskip
\noindent{\it (iv)} Permutation graphs are competitivity graphs generated by two rankings, and are also comparability graphs as proved in \cite{DuMi} and  \cite{BraLeSpi}. Conversely, suppose that $G$ is a competitivity and  a comparability graph. Since it is  a competitivity graph, it is  semi-cohesive. With respect to that labeling of the nodes, condition {\it (ii)} of Definition~\ref{def:filipina} of cohesive-vertex set order also holds because the graph  has a transitive orientation of its vertices. Then by Theorem~\ref{th:filipino} (see \cite{GeRaRa}, Theorem 2.3), $G$ is a permutation graph. In addition to this, note that it is easy to check that $C_4$ is a permutation but not a chordal graph.

\medskip
\noindent{\it (v)} The graph $S_3$, shown in Figure~\ref{fig:example}.a,  is a chordal graph but it can be checked that it is not semi-cohesive (by using Lemma~\ref{commoneighbours} and following similar reasonings to those used in the proof of Lemma~\ref{pentagononocomp}) and hence it is not a competitivity graph. Moreover, it does not admit a transitive orientation of its edges.

\bigskip
\begin{figure}[ht!]
\begin{center}
\begin{tabular}{cc}
  \includegraphics{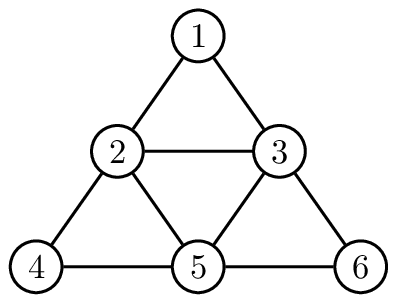} & \includegraphics{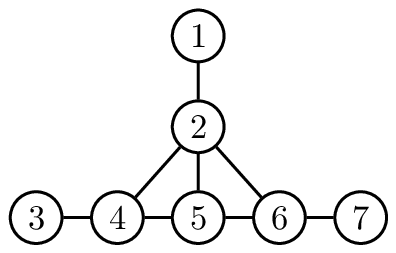} \\
  $\enspace$ & $\enspace$ \\
  (a) & (b) \\
  $\enspace$ & $\enspace$ \\
  $\enspace$ & $\enspace$ \\
  \includegraphics{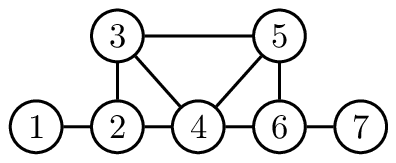} & \includegraphics{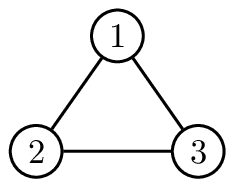} \\
  $\enspace$ & $\enspace$ \\
  (c) & (d)
\end{tabular}
\end{center}
\caption{\label{fig:example} The graphs $S_3$ (Figure~\ref{fig:example}.a) that it is chordal but not competitivity not comparability graph, $XF_{2}^2$ (Figure~\ref{fig:example}.b),  that it is chordal and comparability but not competitivity graph, $X_{176}$ (Figure~\ref{fig:example}.c), that it is chordal and competitivity, but not comparability graph, and $C_3$ (Figure~\ref{fig:example}.d), that it is chordal and permutation graph.}
\end{figure}

\medskip
\noindent{\it (vi)} The graph $X_{176}$, shown in Figure~\ref{fig:example}.c, is chordal, but it is not comparability since it does not admit a transitive orientation of its edges. It is a competitivity graph because it can be generated by the family of rankings ${\mathcal R}=\{c_1,c_2,c_3,c_4,c_5\}$ given by
\begin{align*}
c_1\equiv\enspace &(1,2,3,4,5,6,7),\\
c_2\equiv\enspace &(2,1,4,3,6,5,7),\\
c_3\equiv\enspace &(1,2,5,4,3,7,6),\\
c_4\equiv\enspace &(1,4,2,3,5,6,7),\\
c_5\equiv\enspace &(1,3,2,6,4,5,7).
\end{align*}

\medskip
\noindent{\it (vii)} The graph $XF_2^2$, shown in Figure~\ref{fig:example}.b, is a comparability and chordal graph. Nevertheless it is not a semi-cohesive graph (by using Lemma~\ref{commoneighbours} and following similar reasonings to those used in the proof of Lemma~\ref{pentagononocomp}) and hence it is not a competitivity graph.

\medskip
\noindent{\it (viii)} The cycle of 3 vertices $C_3$, shown in Figure~\ref{fig:example}.d, is chordal, and a permutation graph (hence comparability and competitivity). Indeed it is generated by the family of rankings ${\mathcal R}=\{c_1,c_2\}$ given by
\begin{align*}
c_1\equiv\enspace &(1,2,3),\\
c_2\equiv\enspace &(3,2,1).
\end{align*}
\end{proof}

\begin{remark}
We can summarize Theorem~\ref{contraejemplos} in Figure~\ref{fig:resumen}. Each square in this Figure corresponds to a class of graphs, as follows:
\begin{itemize}
 \item[{\it (i)}] $\mathcal G$ is the set of all finite graphs.
 \item[{\it (ii)}] $C{\mathcal G}$ is the set of all comparability finite graphs.
 \item[{\it (iii)}] $CH{\mathcal G}$ is the set of all chordal graphs.
 \item[{\it (iv)}] $CP{\mathcal G}$ is the set of all competitivity graphs.
\end{itemize}
Then, Theorem~\ref{contraejemplos} shows that each region in the Figure is non-empty (i.e. there some graphs in each region) and $C{\mathcal G}\cap CP{\mathcal G}=P{\mathcal G}$.

\bigskip
\begin{figure}[ht!]
\centerline{\includegraphics[width=0.4\textwidth]{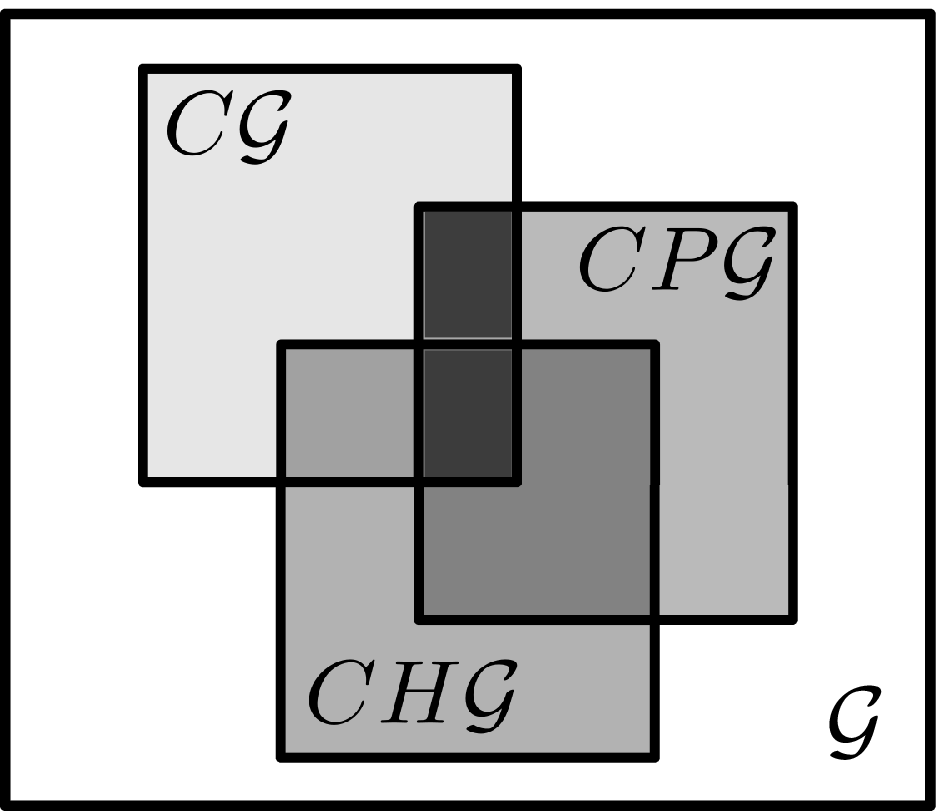}}
\bigskip
\caption{\label{fig:resumen} A visual summary of Theorem~\ref{contraejemplos}. Each region in the figure is non-empty and $C{\mathcal G}\cap CP{\mathcal G}=P{\mathcal G}$.}
\end{figure}
\end{remark}
 \section{Connected Components of the competitiveness graph: Structural properties and computation}\label{sec:components}


In this section we will make a deeper study of sets of eventual competitors. This will lead to a new algorithm to calculate them without computing the competitivity graph in advance. Moreover, we will show that sets of eventual competitors can be ordered by means of a total order relation, giving rise to the leader and the looser sets of eventual competitors.

First of all, we will show a property related to the {\sl convexity} with respect of the rankings of the sets of eventual competitors:

\begin{lemma}\label{convexity}
Given a finite set ${\mathcal R}=\{c_1, c_2,\dots, c_r\}$ of rankings of nodes ${\mathcal N}=\{1,\dots,n\}$, if $D\subseteq {\mathcal N}$ is a set of eventual competitors and $a,b\in D$, then for every $x\in {\mathcal N}$ and every ranking $c_m\in{\mathcal R}$ such that
\[
a\prec_{c_m} x\prec_{c_m} b \implies x\in D.
\]
\end{lemma}

\begin{proof} \underbar{Step 1}: Suppose first that the pair of nodes $(a,b)$ compete. Then, since the competitivity graph is semi-cohesive  (\ref{basicproperty}), $(a,x)$ compete or $(x,b)$ compete, so $x\in D$.

\noindent\underbar{General case}: Since $a,b\in D$, they are eventual competitors and there exist  $k\in\mathbb{N}$ and nodes $i_i,\dots, i_k\in {\mathcal N}$ such that $(a, i_1)$ compete, $(i_1,i_2)$ compete,\dots, and $(i_k,b)$ compete. If $a\prec_{c_m}x\prec_{c_m}i_1$ we can apply Step 1 to nodes $a,x$ and $i_1$ since $(a,i_1)$ compete, and we get that $x\in D$. Otherwise $a\prec_{c_m}i_1\prec_{c_m}x\prec_{c_m}b$. Replace $a$ by node $i_1$ and repeat the same argument. The result follows by induction.
\end{proof}

In order to compute the sets of eventual competitors without calculating the competitivity graph in advance, we need the following auxiliar lemma.

\begin{lemma}\label{casobase} Given a finite set ${\mathcal R}=\{c_1, \dots, c_r\}$ of rankings of nodes ${\mathcal N}=\{1,\dots,n\}$, if $D\subseteq {\mathcal N}$ is a set of eventual competitors and there exist $a\in D$ and $c_m\in{\mathcal R}$ such that $c_m^{-1}(a)=1$ (element $a$ appears in the first position of $c_m$) then
\[
\{x\in{\mathcal N}\ |\ c_s^{-1}(x)=1\hbox{ for some $c_s\in {\mathcal R}$}\}\subseteq D,
\]
i.e., all elements in the first position of the rankings of ${\mathcal R}$ belong to $D$.
\end{lemma}

\begin{proof} If $c_m\ne c_s$ and $a\neq x$ satisfy  $c_m^{-1}(a)=1=c_s^{-1}(x)$, then clearly $a\prec_{c_m}x$ and $x\prec_{c_s}a$, so $(a,x)$ compete and therefore $x\in D$.
\end{proof}

\begin{theorem}
Given a finite set ${\mathcal R}=\{c_1, \dots, c_r\}$ of rankings of nodes ${\mathcal N}=\{1,\dots,n\}$, the sets of eventual competitors can be identified with closed intervals of natural numbers $[p,q]$ in the following sense
\[
 D_{[p,q]}= \{x\in{\mathcal N}\ |\ c_s^{-1}(x)\in[p,q]\hbox{ for some $c_s\in {\mathcal R}$}\}.
 \]
 Moreover $p$ and $q$ are the first on the left  and last on the right  positions of the elements of $D_{[p,q]}$ with respect to all rankings.
\end{theorem}

\begin{proof} We will show that every set of eventual competitors has the form $D_{[p,q]}$ for some natural numbers $p,q$.
Let $a\in {\mathcal N}$, $c_m\in \mathcal{R}$ such that $c_m^{-1}(a)=1$ and let $D$ be the set of eventual competitors that contains $a$. Set
\[
D_{[1,1]}=\{x\in{\mathcal N}\ |\ c_s^{-1}(x)=1\hbox{ for some $c_s\in {\mathcal R}$}\}.
\]
By Lemma~\ref{casobase}, we get that $D_{[1,1]}\subseteq D$.

Now, define
\[
D_{[1,p_k]}=\{x\in{\mathcal N}\ |\ c_s^{-1}(x)\in[1, p_k]\hbox{ for some $c_s\in{\mathcal R}$}\}
\]
and set $p_{k+1}$ as the last (to the right) position of any element of $D_{[1,p_k]}$ with respect to any ranking. We claim that if $D_{[1,p_k]}\subseteq D$  and
\[
D_{[1,p_{k+1}]}=\{x\in{\mathcal N}\ |\ c_s^{-1}(x)\in[1, p_{k+1}]\hbox{ for some $c_s\in{\mathcal R}$}\},
\]
then $D_{[1,p_{k+1}]}\subseteq D$. Let $x\in{\mathcal N}$ with $c_s^{-1}(x)\in [1,p_{k+1}]$ for some $c_s$; then either  $c_s^{-1}(x)\in [1,p_k]$ and $x\in D_{[1,p_k]}\subset D$ by hypothesis, or $c_s^{-1}(x)\in [p_k,p_{k+1}]$. In this second case, let $b$ be an element of $D_{[1,p_k]}$ who appears in position $p_{k+1}$  with respect to some ranking $c_{m_b}$, i.e., $c_{m_b}^{-1}(b)=p_{k+1}$. If $x\prec_{c_{m_b}}b$ then $x\in D$ by Lemma~\ref{convexity}, so we can suppose that $b\prec_{c_{m_b}}x$. All the elements on the left of $b$ in ranking $c_{m_b}$ belong to $D$ by Lemma~\ref{convexity}. Suppose that the number of those elements is $t$. If $x\prec_{c_s}b$ then $(x,b)$ compete and $x\in D$, so we can suppose that $b\prec_{c_s}x$. On the left of $x$ in ranking $c_s$ there are at most $t$ elements but one of them is $b$ so one of the elements $z$ such that $z\prec_{c_{m_b}} b$  must go after $x$ in ranking $c_s$, making $(x,z)$ compete, thus $x\in D$.

Since ${\mathcal N}$ is finite and $D_{[1,p_m]}\subseteq {\mathcal N}$, the chain of sets $D_{[1,1]}\subset D_{[1,p_1]}\subset...$ stabilizes at some $D_{[1,p_m]}\subseteq D$. Moreover, $D\subseteq D_{[1,p_m]}$: by hypothesis $a\in D$ so given any other element $x\in D$ there exist a finite number of elements $a_1,a_2,\dots,a_k$ such that $(a,a_1)$, $(a_1,a_2)$,\dots, $(a_k,x)$ compete. From the facts that $a\in D_{[1,1]}$ and $(a,a_1)$ compete we get that $a_1\in D_{[1,p_1]}$; similarly, from $a_1\in D_{[1,p_1]}$ and $(a_1,a_2)$ compete, $a_2\in D_{[1,p_2]}$,... and repeating this argument we get $x\in D_{[1,p_{k+1}]}\subseteq D_{[1,p_m]}$.

Delete from ${\mathcal N}$ all the elements appearing in $D$ and repeat the same argument to find the rest of sets of eventual competitors.
\end{proof}

\begin{remark}
Since the proof of last theorem is constructive it gives an algorithm for computing the sets of eventual competitors directly from the rankings and not as connected components of the competitivity graph.

\begin{algorithm}
 \SetKwInOut{Input}{Input}
 \SetKwInOut{Output}{Output}
 \BlankLine
 \Input{\begin{itemize}
           \item A finite set of nodes ${\mathcal N}=\{1,\dots, n\}$ ($n\in \N$)
           \item A finite set of rankings ${\mathcal R}=\{c_1,\dots, c_r\}$ of ${\mathcal N}$ ($r\in \N$)
         \end{itemize}}
 \BlankLine
 \Begin{
  $j:=1$\;
  $q_0=0$\;
  $q_j:=1$\;
  \While{$|{\mathcal N}|>0$}{
    $D_j:=\emptyset$\;
    $p_0:=q_{j-1}$\;
    $p_1:=q_j$\;
    $i:=0$\;
    \While{$p_i\ne p_{i+1}$}{
     $i:=i+1$\;
     \emph{Construct $D_j:=D_{[q_j,p_i]}$}\;
     $\displaystyle p_{i+1}:=\max_{x\in D_j,\, c\in{\mathcal R}} c^{-1}(x)$\;
     }
    ${\mathcal N}:={\mathcal N}\setminus D_j$\;
    $j:=j+1$\;
    $q_j:=p_i$\;
  }
  }
 \BlankLine
 \Output{Sets of eventual competitors $D_1,\dots,D_k$}
 \caption{Computation of sets of eventual competitors}
 \BlankLine
\end{algorithm}

It is easy to check that the computational cost of this algorithms is of order $kn^2$. Note that since the sets of eventual competitors coincide with the connected components of the competitivity graph, they can also be calculated with the usual connected components algorithms (which finally has computational complexity of order $kn^2$ as well), but this method requires the computation of the competitivity graph in advance.
\medskip
\end{remark}


In the last part of this section we will show  that we can define a binary relation between sets of eventual competitors that is a total order relation. Before that we need to introduce a directed graph related to a set of rankings, included in the following definition.

\begin{definition}
Given a set  ${\mathcal R}=\{c_1,\dots, c_r\}$ of $r$ rankings ($r\ge 2$) of nodes ${\mathcal N}=\{1,\dots, n\}$, we define the directed graph $G_{d}({\mathcal R})$ in the following way:
\begin{itemize}
 \item[{\it (i)}] the vertex set of $G_{d}({\mathcal R})$ is ${\mathcal N}$.
 \item[{\it (ii)}] if $i\ne j\in{\mathcal N}$ there is an outlink from $i$ to $j$ in $G_{d}({\mathcal R})$ if there exists a ranking $c_m\in {\mathcal R}$ such that $i \preceq_{c_m}  j$.
\end{itemize}
\end{definition}

\begin{remark}
Notice that this directed graph $G_{d}({\mathcal R})$ coincides with the directed graph $G_{\preceq}$  defined by the (reflexive and antisymmetric) relation $\preceq$ given by:
 \begin{itemize}
  \item[{\it (i)}] $i\preceq i$ for any $i\in {\mathcal N}$,
  \item[{\it (ii)}] $i\preceq j$ ($i\neq j\in {\mathcal N}$)  if there exists $c_m\in {\mathcal R}$ such that $i\preceq_{c_m} j$. 
 \end{itemize}
The competitivity graph $G_c({\mathcal R})$ coincides with the undirected graph with the same set of nodes of $G_{d}({\mathcal R})$ and links between  pairs of nodes $(i,j)$ when there is an outlink from $i$ to $j$ and an outlink from $j$ to $i$ in  $G_{d}({\mathcal R})$.
\end{remark}

\begin{proposition}\label{binaryrelation}
Following the same notation as before, if $D_1$ and $D_2$ are two different sets of eventual competitors, the following conditions about the directed graph $G_d({\mathcal R})$ are equivalent:
 \begin{itemize}
 \item[{\it(i)}] there is an outlink from a node $a\in D_1$ to a node $b\in D_2$,
 \item[{\it (ii)}] for every node in $D_1$ there is an outlink to every node of $D_2$
 \end{itemize}(similarly if we replace the word {\rm outlink} with the word {\rm inlink}).
\end{proposition}

\begin{proof} We will separate the proof in three steps:

\noindent\underbar{Step 1}: We will show that if $a\in D_1$,  $b_1,b_2\in D_2$, the pair $(b_1,b_2)$ compete and there is an outlink from $a$ to $b_1$, then necessarily there is an outlink from $a$ to $b_2$ (similarly if we replace the word {\it outlink} by {\it inlink}). By hypothesis there exists a ranking $c_m$ such that $a\prec_{c_m}b_1$. If $a\prec_{c_m}b_2$ the claim holds; otherwise  $b_2\prec_{c_m} a\prec_{c_m}b_1$, but since $(b_1,b_2)$  compete  there exists another ranking $c_{m'}$ such that $b_1\prec_{c_{m'}}b_2$ and since $a$ cannot compete with $b_1$ necessarily $a\prec_{c_{m'}} b_1\prec_{c_{m'}}b_2$  making $(a,b_2)$ compete (a contradiction).

\noindent\underbar{Step 2}: We will show that if $a\in D_1$, $b\in D_2$ and there is an outlink from $a$ to $b$ then for every $b'$ of $D_2$ there is an outlink from $a$ to $b'$ (similarly if we replace the word {\it outlink} by {\it inlink}). Since $b,b'\in D_2$ there exist $k\in \mathbb{N}$ and $b_1,\dots, b_k\in D_2$ such that $(b,b_1)$ compete, $(b_1,b_2)$ compete, \dots, $(b_k,b')$ compete. Nodes $a,b,b_1$ are in the conditions of Step 1 so there is an outlink from $a$ to $b_1$; again nodes $a,b_1,b_2$ are in the conditions of Step 1  so there is an outlink from $a$ to $b_2$, \dots, and the result follows by induction.

\noindent\underbar{Step 3}: Proof of $(i)\Rightarrow (ii)$. If from an element $a\in D_1$ there is an outlink to an element in $D_2$, by Step 2 there is an outlink from $a$ to any element of $D_2$. Fix now any element of $D_2$ and use Step 2 to get that there are outlinks from any element of $D_1$ to this fixed element.
\end{proof}

As a consequence of Proposition~\ref{binaryrelation}, we can define a binary relation between the sets of eventual competitors as follows.

\begin{definition}\label{def:binaryrelation}
Given a set  ${\mathcal R}=\{c_1,\dots, c_r\}$ of $r$ rankings ($r\ge 2$) of nodes ${\mathcal N}=\{1,\dots, n\}$ whose sets of eventual competitors are denoted by $D_1,\dots, D_k$, we define a binary relation $\cdot\longrightarrow\cdot$ between sets of eventual competitors as follows:
 \begin{itemize}
\item[{\it (i)}] $D_i\longrightarrow D_i$ for every set of eventual competitors $D_i$ of ${\mathcal N}$,
\item[{\it (ii)}] for any two different sets of eventual competitors $D_i, D_j$,
\[
D_i\longrightarrow D_j \iff \hbox{ any of the two statements of Proposition~\ref{binaryrelation} holds.}
\]
\end{itemize}
\end{definition}

\begin{lemma}
Following the same notation as before, the binary relation given  in Definition~\ref{def:binaryrelation} is transitive, i.e., if $D_1$, $D_2$, $D_3$ are three different sets of eventual competitors and $D_1\longrightarrow D_2$ and $D_2\longrightarrow D_3$ then $D_1\longrightarrow D_3$.
\end{lemma}

\begin{proof}
Suppose that $D_1\longrightarrow D_2$, $D_2\longrightarrow D_3$ but $D_3\longrightarrow D_1$. Take a node $x\in D_1$. Since  $D_3\longrightarrow D_1$ there exists a ranking $c_m$ such that $a\prec_{c_m}x$ for all $a\in D_3$. Moreover, since $D_1\longrightarrow D_2$, $x\prec_{c_m}b$ for all $b\in D_2$, and therefore $a\prec_{c_m} b$ for all $a\in D_3$ and all $b\in D_2$, i.e., $D_3\longrightarrow D_2$ leading to a contradiction.
\end{proof}

\begin{corollary}
The binary relation given in Definition~\ref{def:binaryrelation} is a total order relation between sets of eventual competitors of ${\mathcal N}$.
\end{corollary}

\begin{remark}
With respect to this total order we can define a directed graph whose nodes are the sets of eventual competitors and edges are given by the rule: there is a link from node $D_i$ to node $D_j$ if $D_i\longrightarrow D_j$. The node $D_i$ with outlinks to any other node will be called the {\it leader among the sets of eventual competitors}, and the node with inlinks from any other node will be called the {\it looser  among the sets of eventual competitors}. Indeed, we obtain a ranking of sets of eventual competitors: the first one, the second, etc.
\end{remark}



\section*{Acknowledgements}

This paper was partially supported by Spanish MICINN Funds and FEDER Funds MTM2009-13848, MTM2010-16153 and MTM2010-18674, and Junta de Andalucia Funds FQM-264.



\end{document}